\documentclass{amsart}

\newcommand{\Ra}{\Rightarrow}
\newcommand{\IR}{\mathbb R}
\newcommand{\IQ}{\mathbb Q}
\newcommand{\w}{\omega}
\newcommand{\LI}{LI}
\newcommand{\SI}{SI}
\newcommand{\K}{\mathcal K}
\newcommand{\F}{\mathcal F}
\newcommand{\U}{\mathcal U}
\newcommand{\V}{\mathcal V}

\newtheorem{theorem}{Theorem}
\newtheorem{corollary}{Corollary}
\newtheorem{example}{Example}
\theoremstyle{definition}
\newtheorem{remark}{Remark}

\title{The continuity properties of compact-preserving functions}
\author{Taras Banakh, Artur Bartoszewicz,  Marek Bienias, Szymon G\l\c ab}
\address{T.Banakh: Ivan Franko University of Lviv (Ukraine) and Jan Kochanowski Uniwersity in Kielce (Poland)}
\email{t.o.banakh@gmail.com}
\address{A.Bartoszewicz,  M.Bienias, S.G\l\c ab: Institute of Mathematics, Technical University of \L\'od\'z, W\'olcza\'nska 215, 93-005
\L\'od\'z, Poland}
\email{arturbar@p.lodz.pl, marek.bienias88@gmail.com, szymon.glab@p.lodz.pl}
\thanks{The first author has been partially financed by NCN grant  DEC-2011/01/B/ST1/01439. The second author has been supported by the Polish Ministry of Science and Higher Education Grant No N N201 414939 (2010-2013). The last two authors was supported by the Polish Ministry of Science and Higher Education Grant No IP2011 014671 (2012-2014).}
\begin{document}
\begin{abstract} A function $f:X\to Y$ between topological spaces is called {\em compact-preserving} if the image $f(K)$ of each compact subset $K\subset X$ is compact. We prove that a function $f:X\to Y$ defined on a strong Fr\'echet space $X$ is compact-preserving if and only if for each point $x\in X$ there is a compact subset $K_x\subset Y$ such that for each neighborhood $O_{f(x)}\subset Y$ of $f(x)$ there is a neighborhood $O_x\subset X$ of $x$ such that $f(O_x)\subset O_{f(x)}\cup K_x$ and the set $K_x\setminus O_{f(x)}$ is finite. This characterization is applied to give an alternative proof of a classical characterization of continuous functions on locally connected metrizable spaces as functions that preserve compact and connected sets. Also we show that for each compact-preserving function $f:X\to Y$ defined on a (strong) Fr\'echet space $X$, the restriction $f|\LI'_f$ (resp. $f|\LI_f)$ is continuous. Here $\LI_f$ is the set of points $x\in X$ of local infinity of $f$ and $\LI'_f$ is the set of non-isolated points of the set $\LI_f$. Suitable examples show that the obtained results cannot be improved.
\end{abstract}

\maketitle

It is known that a function $f:\IR\to \IR$ is continuous if and only if $f$ is {\em preserving} in the sense that for each compact subset $K\subset \IR$ the image $f(K)$ is compact and for each connected subset $C\subset\IR$ the image $f(C)$ is connected. The first result of this sort has appeared in 1926 \cite{Rowe} and later was rediscovered and generalized by many authors: \cite{Whyburn27}, \cite{KU}, \cite{Halfar}, \cite{Whyburn65}, \cite{White68},  \cite{McM}, \cite{White71}, \cite{Vel}, \cite{AP}, \cite{GJSS}, \cite{GMS}.

In this paper we study what remains of the continuity of a function $f:X\to Y$ if it merely preserves compact sets. Of course such a function can be everywhere discontinuous as shown by the classical Dirichlet function $\delta_\IQ:\IR\to\{0,1\}$ equal 1 on rationals and 0 on irrationals. In this case the compact-preserving property follows from the local finity.

We define a function $f:X\to Y$ between topological spaces to be
\begin{itemize}
\item {\em compact-preserving} if the image $f(K)$ of each compact subset $K\subset X$ is compact;
\item {\em locally finite} at a point $x\in X$ if for some neighborhood $O_x\subset X$ of $x$ the image  $f(O_x)$ is finite;
\item {\em locally infinite} at $x\in X$ if $f$ is not locally finite at $x$;
\item {\em sequentially infinite} at $x\in X$ if there is a sequence $\{x_n\}_{n\in\w}$ that converges to $x$ and has infinite image $\{f(x_n)\}_{n\in\w}$.
\end{itemize}
By $\LI_f$ and $\SI_f$ we denote the sets of points $x\in X$ at which the function $f$ is  locally infinite and sequentially infinite, respectively. It is clear that $\SI_f\subset\LI_f$. We shall show that for a compact-preserving function $f:X\to Y$ defined on a Fr\'echet space $X$ the restriction $f|\SI_f$ is continuous and the set $\SI_f$ contains the set $\LI'_f$ of all non-isolated points of $\LI_f$.

We recall that a topological space $X$ is
\begin{itemize}
\item {\em first countable} if each point $x\in X$ possesses a countable base of neighborhoods;
\item {\em Fr\'echet} if for each subset $A\subset X$ and a point $a\in\bar A$ from its closure there is a sequence $\{a_n\}_{n\in\w}\subset A$ that converges to $a$;
\item {\em strong Fr\'echet} if for any decreasing sequence $A_n$, $n\in\w$, of subsets of $X$ and a point $a\in\bigcap_{n\in\w}\bar A_n$ there is a sequence of points $a_n\in A_n$, $n\in\w$, that converges to $a$.
\item {\em sequential} if for each non-closed subset $A\subset X$ there is a sequence $\{a_n\}_{n\in\w}\subset A$ that converges to a point $a\in X\setminus A$;
\end{itemize}
These notions relate as follows:
$$\mbox{first countable $\Ra$ strong Fr\'echet $\Ra$ Fr\'echet $\Ra$ sequential.}$$
By \cite[2.4.G]{En}, a topological space $X$ is Fr\'echet if and only if each subspace of $X$ is sequential.
For a function $f:X\to Y$ between topological spaces and a point $x\in X$ consider the set
$$f[x]=\{y\in Y:x\in\mathrm{cl}_X(f^{-1}(y))\}=\textstyle{\bigcap}\{f(O_x):O_x\mbox{ is a neighborhood of $x$ in $X$}\},$$
which can be interpreted as the oscillation of $f$ at $x$. If $f$ is continuous at $x$ and $Y$ is a $T_1$-space, then the  set $f[x]$ coincides with the singleton $\{f(x)\}$.

The following theorem implies that for a compact-preserving function $f:X\to Y$ from a Fr\'echet space $X$ to a Hausdorff space $Y$, the continuity of $f$ at a point $x\in X$ is equivalent to the equality $f[x]=\{f(x)\}$.

\begin{theorem}\label{t1} If $f:X\to Y$ is a compact-preserving function from a (strong) Fr\'echet space $X$ to a Hausdorff topological space $Y$, then for each point $x\in X$ and a neighborhood $O_{f(x)}$ of $f(x)$ in $Y$ there is a neighborhood $O_x$ of $x$ in $X$ such that $f(O_x)\subset f[x]\cup O_{f(x)}$ (and the set $f[x]\setminus O_{f(x)}$ is finite).
\end{theorem}

\begin{proof} Assume conversely that there is a neighborhood $O_{f(x)}\subset Y$ of $f(x)$ such that
$f(O_x)\not\subset f[x]\cup O_{f(x)}$ for each neighborhood $O_x\subset X$ of $x$. This means that the set $A=f^{-1}\big(Y\setminus (f[x]\cup O_{f(x)})\big)$ contains the point $x$ in its closure. By the Fr\'echet property of $X$, the set $A$ contains a sequence $(x_n)_{n\in\w}$ that converges to the point $x$. Replacing this sequence by a suitable subsequence we can assume that either the set $\{f(x_n)\}_{n\in\w}$ is a singleton or else $f(x_n)\ne f(x_m)$ for all $n\ne m$. In the first case the singleton $\{f(x_0)\}=\{f(x_n)\}_{n\in\w}$ lies in the set $f[x]$, which is not possible as $x_0\in A\cap f^{-1}(f[x])=\emptyset$. So, we have the second option which implies that the set $K=\{f(x)\}\cup\{f(x_n)\}_{n\in\w}$ is infinite. Since the function $f$ is compact-preserving, the set $K$ is compact and being infinite, has an accumulating point $y\in K$, which is not equal to $f(x)$ (as $f(x)$, being a unique point of the intersection $K\cap O_{f(x)}$, is isolated in $K$). Then the preimage $f^{-1}(y)$ does not contain the point $x$ and hence the set $S=\{x\}\cup\{x_n\}_{n\in\w}\setminus f^{-1}(y)$ is compact. On the other hand, its image $f(S)=K\setminus\{y\}$ is not compact, which contradicts the compact-preserving property of $f$.

Next, assuming that the space $X$ is strong Fr\'echet, we shall prove that the complement $f[x]\setminus O_{f(x)}$ is finite. Assume conversely that this complement contains a sequence $\{y_n\}_{n\in\w}$ of pairwise distinct points. For every $n\in\w$ consider the set $A_n=\bigcup_{m\ge n}f^{-1}(y_m)$ and observe that $x\in\bigcap_{n\in\w}\bar A_n$. Since $X$ is strong Fr\'echet, there is a sequence of points $a_n\in A_n$, $n\in\w$, which converges to $x$. The compact-preserving property of $f$ guarantees that the set $K=\{f(x)\}\cup\{f(a_n)\}_{n\in\w}\subset\{f(x)\}\cup\{y_n\}_{n\in\w}$ is compact. Being infinite, this set has an accumulation point $y\in K$, which is not equal to the isolated point $f(x)$ of $K$. Then the set $S=\{x\}\cup\{a_n\}_{n\in\w}\setminus f^{-1}(y)$ is compact while its image $f(S)=K\setminus\{y\}$ is not. But this contradicts the compact-preserving property of $f$.
\end{proof}

Theorem~\ref{t1} implies the following characterization of compact-preserving functions defined on strong Fr\'echet spaces.

\begin{corollary} A function $f:X\to Y$ from a (strong Fr\'echet) space $X$ to a Hausdorff space $Y$ is compact-preserving if (and only if) for each point $x\in X$ there is a subset $K_x\subset Y$ such that for each neighborhood $O_{f(x)}\subset Y$ of $f(x)$ there is a neighborhood $O_x\subset X$ of $x$ such that $f(O_x)\subset O_{f(x)}\cup K_x$ and $K_x\setminus O_{f(x)}$ is finite.
\end{corollary}

\begin{proof} The ``only if''part follows from Theorem~\ref{t1}.

To prove the ``if'' part, assume that for each point $x\in X$ there is a subset $K_x\subset Y$ such that for each neighborhood $O_{f(x)}\subset Y$ of $f(x)$ there is a neighborhood $O_x\subset X$ of $x$ such that $f(O_x)\subset O_{f(x)}\cup K_x$ and $K_x\setminus O_{f(x)}$ is finite. Given a compact subset $K\subset X$ we need to prove that the image $f(K)$ is compact. Let $\U$ be a cover of $f(K)$ by open subsets of $Y$. For each point $x\in X$ find an open set $O_{f(x)}\in\U$ that contains the point $f(x)\in f(K)$. Next, find an open neighborhood $O_x\subset X$ of $x$ such that $f(O_x)\subset O_{f(x)}\cup K_x$. The finiteness of the set $K_x\setminus O_{f(x)}$ implies that the set $f(O_x)\setminus O_{f(x)}$ is finite.

The compactness of the set $K$ guarantees that the open cover $\{O_x:x\in K\}$ of $K$ contains a finite subcover $\{O_x:x\in F\}$. Choose a finite subfamily $\F\subset\U$ whose union $\cup\F$ contains  the finite subset $\bigcup_{x\in F}f(K\cap O_x)\setminus O_{f(x)}$ of $f(K)\subset\bigcup\U$. Then $\V=\F\cup\{O_{f(x)}:x\in F\}$ is the required finite subcover of $f(K)$ witnessing that $f(K)$ is compact.
\end{proof}

Theorem~\ref{t1} can be applied to give an alternative proof of the mentioned ``preserving'' characterization of continuous functions. We recall that a function $f:X\to Y$ has the {\em Darboux property} if the image $f(C)$ of each connected subset $C\subset X$ is connected.

\begin{corollary}\label{c2} A function $f:X\to Y$ from a locally connected strong Fr\'echet space $X$ to a Hausdorff space $Y$ is continuous if and only $f$ is compact-preserving and has the Darboux property.
\end{corollary}

\begin{proof}  The ``only if'' part is trivial. To prove the ``if'' part, assume that a function $f:X\to Y$ is compact-preserving and has the Darboux property. Assuming that the function $f$ is discontinuous at some point $x\in X$, and applying Theorem~\ref{t1}, we conclude that the set $f[x]$ contains a point $y$, distinct from $f(x)$. Since $Y$ is Hausdorff, we can find disjoint open neighborhoods
$O_y$ and  $O_{f(x)}$ of the points $y$ and $f(x)$ in $Y$, respectively.
By Theorem~\ref{t1}, the complement $F=f[x]\setminus O_{f(x)}$ is finite. Replacing the neighborhood $O_y$ by $O_y\setminus (F\setminus\{y\})$, we can assume that $O_y\cap(f[x]\cup O_{f(x)})=\{y\}$, which means that $y$ is an isolated point of the space $f[x]\cup O_{f(x)}$.

By Theorem~\ref{t1}, there is a neighborhood $O_x\subset X$ of $x$ such that $f(O_x)\subset f[x]\cup O_{f(x)}$. Since the space $X$ is locally connected, we can choose the neighborhood $O_x$ to be connected. By the Darboux property, the image $f(O_x)$ is connected and contains the point $y$ by the definition of the set $f[x]\ni y$. On the other hand, the point $y$ is isolated in $f(O_x)$, which implies that $f(O_x)$ cannot be connected. This contradiction completes the proof.
\end{proof}

\begin{remark} In fact, the characterization of the continuity given in Corollary~\ref{c2} holds for any function defined on a locally connected Fr\'echet Hausdorff space, see \cite{McM}.
\end{remark}

Now we shall study the continuity properties of compact-preserving functions in more details.
Let $f:X\to Y$ be a function between topological spaces. A sequence $(x_n)_{n\in\w}$ of points of $X$ will be called
\begin{itemize}
\item {\em injective} if $x_n\ne x_m$ for any distinct numbers $n,m\in\w$;
\item {\em $f$-injective} if $f(x_n)\ne f(x_m)$ for any distinct numbers $n,m\in\w$;
\item {\em $f$-constant} if $\{f(x_n)\}_{n\in\w}$ is a singleton.
\end{itemize}
It is clear that each sequence $(x_n)_{n\in\w}$ in $X$ with finite (resp. infinite) image $\{f(x_n)\}_{n\in\w}$ contains an $f$-constant (resp. $f$-injective) subsequence. Consequently, each sequence in $X$ contains an $f$-constant or $f$-injective subsequence.

The following theorem proved in \cite[Lemma 2]{McM} shows that $f$-injective convergent sequences do not see the discontinuity of compact-preserving functions.

\begin{theorem}\label{t2} For any compact-preserving function $f:X\to Y$ from a topological space $X$ to a Hausdorff space $Y$ and each $f$-injective sequence $\{x_n\}_{n\in\w}\subset X$ that converges to a point $x\in X$ the sequence $(f(x_n))_{n\in\w}$ converges to the point $f(x)$.
\end{theorem}

\begin{proof} The set $K=\{x\}\cup\{x_n\}_{n\in\w}$ is compact and so is its image $f(K)$. Since the sequence $(x_n)_{n\in\w}$ is $f$-injective, the convergence of $(f(x_n))_{n\in\w}$ to $f(x)$ will follow as soon as we check that $f(x)$ is a unique non-isolated point of the compact space $f(K)$. Assuming that $y\ne f(x)$ is another non-isolated point of $f(K)$, we observe that the set $C=K\setminus f^{-1}(y)$ is compact but its image $f(C)=f(K)\setminus\{y\}$ is not, which is impossible as $f$ is compact-preserving.
\end{proof}

Now we establish the promised continuity of a compact-preserving function $f$ on the set $\SI_f$ of all points $x\in X$ at which $f$ is sequentially infinite.

\begin{theorem}\label{t3} For each compact-preserving function $f:X\to Y$ from a Fr\'echet space $X$ to a Hausdorff space the restriction $f|\SI_f$ is continuous.
\end{theorem}

\begin{proof} Assume that $f|\SI_f$ is discontinuous at some point $x\in \SI_f$. In this case we can find a neighborhood $O_{f(x)}$ of $f(x)$ in $Y$ such that the set $A=\SI_f\setminus f^{-1}(O_{f(x)})$ contains the point $x$ in its closure. By the Fr\'echet property of $X$, there is a sequence $\{x_n\}_{n\in\w}\subset A$ that converges to the point $x$. Passing to a suitable subsequence, we can assume that the sequence $(x_n)_{n\in\w}$ is $f$-injective or $f$-constant. Since the sequence $(f(x_n))_{n\in\w}$ does not converge to $f(x)$, it cannot be $f$-injective by Theorem~\ref{t2}.
So, $(x_n)_{n\in\w}$ is $f$-constant and hence $\{f(x_n)\}_{n\in\w}=\{y\}$ for some $y\in Y$. Since $Y$ is Hausdorff, the points $f(x)$ and $y$ have disjoint open neighborhoods $U_{f(x)}\subset O_{f(x)}$ and $U_y\subset Y$, respectively.

For every $n\in\w$ the function  $f$ is sequentially infinite at $x_n$. Consequently, there is a convergent to $x$ sequence $(x_{n,m})_{m\in\w}$ with infinite image $\{f(x_{n,m})\}_{m\in\w}$.
By a standard diagonal argument, we can replace the sequences $(x_{n,m})_{m\in\w}$, $n\in\w$, by suitable subsequences, and assume that the double sequence $(x_{n,m})_{n,m\in\w}$ is {\em $f$-injective} in the sense that $f(x_{i,k})\ne f(x_{j,m})$ for any distinct pairs $(i,k),(j,m)\in\w\times\w$.

By Theorem~\ref{t2}, for every $n\in\w$ the sequence $\big(f(x_{n,m})\big)_{m\in\w}$ converges to $f(x_n)=y$.
Again passing to a subsequence, we can assume that $f(x_{n,m})\in U_y$ for all $m\in\w$.
Now consider the set $\{x_{n,m}\}_{n,m\in\w}$ and observe that it contains the set $\{x\}\cup\{x_n\}_{n\in\w}$ in its closure. Then Fr\'echet property of $X$ yields an injective sequence $\{z_k\}_{k\in\w}\subset \{x_{n,m}\}_{n,m\in\w}$ that converges to $x$. The $f$-injectivity of the double sequence $(x_{n,m})_{n,m\in\w}$ implies the $f$-injectivity of the sequence $(z_k)_{k\in\w}$. By Theorem~\ref{t2}, the sequence $\{f(z_k)\}_{k\in\w}\subset U_y$ converges to $f(x)$, which is not possible as $f(x)$ does not belong to the closure of $U_y$.
\end{proof}

In light of Theorem~\ref{t3} it is important to know the structure of the set $\SI_f$. Let us recall that by $\LI'_f$ we denote the (closed) set of non-isolated points of the (closed) set $\LI_f$ of points $x\in X$ at which the function $f$ is locally infinite.

A topological space $X$ is called {\em sequentially Hausdorff\/} if each convergent sequence in $X$ has a unique limit.

\begin{theorem}\label{t4} For a compact-preserving function $f:X\to Y$ from a sequentially Hausdorff Fr\'echet space $X$ to a Hausdorff space $Y$,
\begin{enumerate}
\item the set $\SI_f$ is closed;
\item $\LI'_f\subset\SI_f\subset\LI_f$;
\item $\SI_f=\LI_f$ provided that $X$ is strong Fr\'echet.
\end{enumerate}
\end{theorem}

\begin{proof} 1. Given any point $x$ in the closure $\overline{\SI}_f$ of $\SI_f$, apply the Fr\'echet property of $X$ and find a sequence $\{x_n\}_{n\in\w}\subset\SI_f$ that converges to $x$. For every $n\in\w$ the definition of the set $\SI_f\ni x_n$ yields an $f$-injective sequence $(x_{n,m})_{m\in\w}$ that converges to $x_n$. By a standard diagonal inductive procedure, we can replace the sequences $(x_{n,m})_{m\in\w}$ by suitable subsequences and assume that the double sequence $(x_{n,m})_{n,m\in\w}$ is $f$-injective and the infinite set $\{f(x_{n,m})\}_{n,m\in\w}$ does not contain $f(x)$. Since the set $\{x_{n,m}\}_{n,m\in\w}\not\ni x$ contains $x$ in its closure, by the Fr\'echet property of $X$, there is an injective sequence $(z_k)_{k\in\w}\subset \{x_{n,m}\}_{n,m\in\w}$ that converges to $x$. The $f$-injectivity of the double sequence $(x_{n,m})_{n,m\in\w}$ implies the $f$-injectivity of the sequence $(z_k)_{k\in\w}$, which witnesses that $f$ is sequentially infinite at $x$ and hence $x\in\SI_f$.
\smallskip

2. The inclusion $\SI_f\subset\LI_f$ trivially follows from the definitions. To prove that $\LI'_f\subset\SI_f$, fix a non-isolated point $x$ of the set $\LI_f$ and using the Fr\'echet property of $X$, find a sequence $\{x_n\}_{n\in\w}\subset \LI_f\setminus\{x\}$ that converges to $x$.

Let $F_{-1}=\{f(x),f(x_0)\}$. By induction, for every $n\in\w$ we shall construct a finite subset $F_n\subset X$ and a sequence $(x_{n,m})_{m\in\w}$ convergent to the point $x_n$ such that
\begin{enumerate}
\item[(a)] $\{f(x_{n,m})\}_{m\in\w}\subset Y\setminus F_{n-1}$;
\item[(b)] the sequence $(x_{n,m})_{m\in\w}$ is either $f$-injective or $f$-constant;
\item[(c)] $F_{n-1}\cup \{f(x_{n+1})\}\subset F_{n}$;
\item[(d)] $\{f(x_{n,m})\}_{m\in\w}\subset F_{n}$ if the sequence $(x_{n,m})_{m\in\w}$ is $f$-constant.
\end{enumerate}

Assume that for some $n\in\w$ the finite set $F_{n-1}$ has been constructed. The condition (c) of the inductive construction guarantees that $f(x_n)\in F_{n-1}$. Since $f$ is locally infinite at $x_n$, the point $x_n$ does not belong to the interior of the set $f^{-1}(F_{n-1})$. Consequently, the set $A_n=X\setminus f^{-1}(F_{n-1})$ contains $x_n$ in its closure. By the Fr\'echet property of $X$, there is a sequence $\{x_{n,m}\}_{m\in\w}\subset A_n$ which converges to $x_n$. Passing to a subsequence we can additionally assume that this sequence is either $f$-injective or $f$-constant. Put $F_{n+1}=F_n\cup\{f(x_{n+1})\}$ if the sequence $(x_{n,m})_{m\in\w}$ is $f$-injective and  $F_n=F_{n-1}\cup\{f(x_{n+1})\}\cup\{f(x_{n,m}):m\in\w\}$, otherwise.

After completing the inductive construction, we obtain an increasing sequence $(F_n)_{n\in\w}$ of finite subsets of $Y$ and $f$-injective or $f$-constant sequences $(x_{n,m})_{m\in\w}$, $n\in\w$,  with $\lim_{m\to\infty} x_{n,m}=x_n$.
Replacing $(x_n)_{n\in\w}$ by a suitable subsequence, we can assume that either for all $n\in\w$ the sequence $(x_{n,m})_{m\in\w}$ is $f$-constant or for all $n\in\w$ this sequence of $f$-injective.

In the second case we can apply a standard inductive diagonal argument and replacing each sequence $(x_{n,m})_{m\in\w}$, $n\in\w$, by a suitable subsequence, assume that the double sequence $(x_{n,m})_{n,m\in\w}$ if $f$-injective.

Since the set $\{x_{n,m}\}_{n\in\w}$ contains the set $\{x\}\cup\{x_n\}_{n\in\w}$ in its closure, the Fr\'echet property of $X$ yields an injective sequence $\{z_k\}_{k\in\w}\subset \{x_{n,m}\}_{n,m\in\w}$ that converges to $x$.

If for all $n\in\w$ the sequence $(x_{n,m})_{m\in\w}$ is $f$-injective, then so is the double sequence $(x_{n,m})_{n,m\in\w}$ and its injective subsequence $(z_k)_{k\in\w}$, which witnesses that $f$ is sequentially infinite at $x$.

Now assume that for all $n\in\w$ the sequence $(x_{n,m})_{m\in\w}$ is $f$-constant.
Since the space $X$ is sequentially Hausdorff and $(z_k)_{k\in\w}$ converges to $x\ne x_n$, $n\in\w$, the set $\big\{n\in\w:\{z_k\}_{k\in\w}\cap \{x_{n,m}\}_{m\in\w}\ne\emptyset\big\}$ is infinite.
By the construction, for any distinct numbers $n,k\in\w$ the singletons $\{f(x_{n,m})\}_{m\in\w}$ and $\{f(x_{k,m})\}_{m\in\w}$ are distinct, which implies that the set $\{f(z_k)\}_{k\in\w}$ is infinite and hence the sequence $(z_k)_{k\in\w}$ witnesses that the function $f$ is sequentially infinite at $x$. So, $x\in\SI_f$.
\smallskip

3. Assuming that $X$ is strong Fr\'echet, we shall prove that each point $x\in\LI_f$ belongs to $\SI_f$. Assume conversely that $x\notin\SI_f$. Let $F_{-1}=\{f(x)\}$. By induction, for every $n\in\w$ we shall construct a finite subset $F_n\subset X$ and an $f$-constant sequence $(x_{n,m})_{m\in\w}$ convergent to the point $x$ such that
\begin{enumerate}
\item $\{f(x_{n,m})\}_{m\in\w}\subset Y\setminus F_{n-1}$;
\item $F_{n}=F_{n-1}\cup \{f(x_{n,m})\}_{m\in\w}$.
\end{enumerate}

Assume that for some $n\in\w$ the finite set $F_{n-1}$ has been constructed. Since $f$ is locally infinite at $x$, the point $x$ does not belong to the interior of the set $f^{-1}(F_{n-1})$. Consequently, the set $X_n=X\setminus f^{-1}(F_{n-1})$ contains $x$ in its closure. By the Fr\'echet property, there is a sequence $\{x_{n,m}\}_{m\in\w}\subset X_n$ which converges to $x$. Moreover, we can assume that this sequence is $f$-injective or $f$-constant. Since $x\notin\SI_f$,
 the sequence $\{x_{n,m}\}_{n\in\w}$ is $f$-constant and we can put  $F_n=F_{n-1}\cup\{f(x_{n,m}):m\in\w\}$.

After completing the inductive construction, we obtain an increasing sequence $(F_n)_{n\in\w}$ of finite subsets of $Y$ and a sequence $(x_{n,m})_{m\in\w}$, $n\in\w$, of $f$-constant sequences that converge to $x$. For every $n\in\w$ put $A_n=\{x_{k,m}:k,m\in\w,\;k\ge n\}$. Since $x\in\bigcap_{n\in\w}\bar A_n$, the strong Fr\'echet property of $X$ yields a sequence of points $a_n\in A_n$, $n\in\w$, which converges to $x$. Since $f(a_n)\in\bigcup_{n\in\w}F_k\setminus \bigcup_{k=0}^{n-1}F_k$, the set $\{f(a_n)\}_{n\in\w}$ is infinite,  witnessing that the function $f$ is sequentially infinite at $x$ and hence $x\in\SI_f$.
\end{proof}

Theorems~\ref{t3} and \ref{t4} imply:

\begin{corollary}\label{c3} Let $f:X\to Y$ be a compact-preserving function from a sequentially Hausdorff space $X$ to a Hausdorff space $Y$.
\begin{enumerate}
\item If $X$ is Fr\'echet, then the restriction $f|\LI'_f$ is continuous.
\item If $X$ is strong Fr\'echet, then the restriction $f|\LI_f$ is continuous.
\end{enumerate}
\end{corollary}

\begin{corollary} For each compact-preserving function $f:X\to Y$ from a sequentially Hausdorff Fr\'echet space $X$ to a Hausdorff space $Y$ there is a point $x\in X$ at which $f$ is locally finite or continuous.
\end{corollary}

\begin{proof} If there exists a point $x\in X\setminus\LI_f$, then $f$ is locally finite at $x$. So, we assume that $\LI_f=X$. If the space $X=\LI_f$ has an isolated point $x\in X$, then $f$ is continuous at $x$. If $X=\LI_f$ has no isolated points, then $\LI'_f=\LI_f=X$ and the function $f=f|\LI'_f$ is continuous according to Corollary~\ref{c3}(1).
\end{proof}

Now we present two examples showing that Theorem~\ref{t3}, \ref{t4} and Corollary~\ref{c3} cannot be improved.

\begin{example} There is a function $f:X\to \IR$ from a Hausdorff Fr\'echet countable space $X$ such that $\SI_f=\LI'_f\ne\LI_f$ and $f|\LI_f$ is discontinuous.
\end{example}

\begin{proof} Consider the space $X=\w^0\cup\w^1\cup\w^3$ endowed with the topology $\tau$ in which
\begin{enumerate}
\item[(a)] each point $(k,n,m)\in \w^3\subset X$ is isolated;
\item[(b)] a set $U\subset X$ is a neighborhood of a point $(k)\in\w^1\subset X$ if and only if $(k)\in U$ and for every $n\in\w$ there a number $m_n\in\w$ such that $(k,n,m)\in U$ for all $m\ge m_n$;
\item[(c)] a set $V\subset X$ is a neighborhood of the point $\emptyset\in\w^0\in X$ if and only if $\emptyset\in V$ and there is a number $k_0\in\w$ such that for all $k\ge k_0$ and $n,m\in\w$ we get $(k)\in V$ and $(k,n,m)\in V$.
\end{enumerate}
It is easy to see that the space $X$ is Fr\'echet but not strong Fr\'echet.

Now define a function $f:X\to \IR$ by the formula
$$f(x)=\begin{cases}
0&\mbox{if $x=\emptyset\in\w^0$}\\
1&\mbox{if $x\in \w^1$}\\
2^{-(k+n)}&\mbox{if $x=(k,n,m)\in\w^3$}
\end{cases}
$$It is easy to check that the function is compact-preserving, $\LI_f=\w^0\cup\w^1$, $\SI_f=\LI_f'=\w^0$, and $f|\LI_f$ is discontinuous at the point $\emptyset\in\w^0$.
\end{proof}

A topological space $X$ is called a {\em $s_\w$-space} if there is a countable family $\K$ of compact metrizable subsets of $X$ generating the topology of $X$ in the sense that a set $U\subset X$ is open if and only if for each compact set $K\in\K$ the intersection $U\cap K$ is open in $K$. It is well-known that $s_\w$-spaces are sequential and normal (even stratifiable).

\begin{example} There is a countable $s_\w$-space $X$ and a compact-preserving function $f:X\to \IR$ which is locally infinite and discontinuous at each point $x\in X$.
\end{example}

\begin{proof} Consider the space $X=\bigcup_{n\in\w}\w^n$ of all finite sequences of natural numbers, endowed with the topology $\tau$ in which a subset $U\subset X$ is a neighborhood of a point $s=(s_0,\dots,s_{n-1})\in\w^n\subset X$ if and only if $s\in U$ and there exists a number $m_0\in\w$ such that for each $m\ge m_0$ the point $s\hat{\phantom{s}}m=(s_0,\dots,s_{n-1},m)$ belongs to $U$.
It follows that the sequence $(s\hat{\phantom{s}}m)_{m\in\w}$ converges to $s$ and hence the set $K_s=\{s\}\cup\{s\hat{\phantom{s}}m\}_{m\in\w}$ is compact. It is easy to see that $X$ is a countable $s_\w$-space whose topology is generated by the countable family $\K=\{K_s:s\in X\}$ of compact subsets of $X$. The space $X$ was first described by Arhangelskii and Franklin \cite{AF} and is known in General Topology as the {\em the Arhangelskii-Franklin space}.

Now consider the function $f:X\to \IR$ assigning to each sequence $s=(s_0,\dots,s_{n-1})\in \w^n\subset X$ the real number
$$f(s)=\begin{cases}
n+1&\mbox{if $n$ is even,}\\
n+2^{-s_{n-1}}&\mbox{if $n$ is odd.}
\end{cases}
$$
It is easy to see that the function $f$ is compact-preserving. On the other hand, each non-empty open set $U\subset X$ has unbounded image $f(U)$ in $\IR$, which implies that $f$ locally infinite and discontinuous at each point $x\in X$.
\end{proof}

\begin{remark} In the proof of our results about compact-preserving functions we did not use the full strength of the compact-preserving property. What we actually used was the compactness of the images of compact subsets with a unique non-isolated point. On the other hand, by transfinite induction it is easy to construct a function $f:[0,1]\to[0,1]$ which maps each uncountable compact subset of $[0,1]$ onto $[0,1]$. Such a function $f$ preserves the compactness of uncountable compact sets and has the Darboux property but is everywhere discontinuous.
\end{remark}

\end{document}